\documentclass{amsart}
\usepackage[english]{babel}
\usepackage[T1]{fontenc}
\usepackage[utf8]{inputenc}
\usepackage{microtype}
\usepackage[autostyle=true]{csquotes}
\usepackage{enumerate}
\usepackage{tikz}
\usepackage[colorlinks=false]{hyperref}
\usepackage{amsmath}
\usepackage{amssymb}
\usepackage{mathrsfs}
\usepackage{amsthm}
\usepackage{leftidx}
\usepackage{cleveref}
\usetikzlibrary{arrows,decorations.markings}
\usetikzlibrary{matrix}
\usetikzlibrary{graphs}
\usetikzlibrary{backgrounds}
\newcommand\Dynkindots{\hbox to 2em{.\hss.\hss.}}
\def\DynkinNodeSize{1.5mm}
\def\DynkinDoubleArrowLength{3.25mm}
\def\DynkinTripleArrowLength{3.5mm}
\tikzset{
  dnode/.style={
    circle,
    inner sep=0pt,
    minimum size=\DynkinNodeSize,
    fill=white,
    draw},
  bnode/.style={
    circle,
    inner sep=0pt,
    minimum size=\DynkinNodeSize,
    fill=black,
    draw},
  middlearrow/.style={
    decoration={markings,
      mark=at position 0.7 with
      {\draw (0:0mm) -- +(+160:\DynkinDoubleArrowLength); \draw (0:0mm) -- +(-160:\DynkinDoubleArrowLength);},
    },
    postaction={decorate}
  },
  triplemiddlearrow/.style={
    decoration={markings,
      mark=at position 0.7 with
      {\draw (0:0mm) -- +(+160:\DynkinTripleArrowLength); \draw (0:0mm) -- +(-160:\DynkinTripleArrowLength);},
    },
    postaction={decorate}
  },	
  sedge/.style={
  },
  dedge/.style={
    middlearrow,
    double distance=1mm,
  },
  tedge/.style={
    triplemiddlearrow,
    double distance=1.0mm+\pgflinewidth,
    postaction={draw},
  },
}
\title[On the values of unipotent characters in type \texorpdfstring{$E_6$}{E6}]{On the values of unipotent characters of finite Chevalley groups of type \texorpdfstring{$E_6$}{E6} in characteristic 3}
\author{Jonas Hetz}
\address{IAZ - Lehrstuhl für Algebra, Universität Stuttgart, Pfaffenwaldring 57, D--70569 Stuttgart, Germany}
\email{jonas.hetz@mathematik.uni-stuttgart.de}
\newtheorem{Thm}{Theorem}[section]
\newtheorem{Prop}[Thm]{Proposition}
\newtheorem{Lm}[Thm]{Lemma}

\theoremstyle{definition}

\newtheorem{Rem}[Thm]{Remark}

\numberwithin{equation}{section}
\newcommand{\mf}{\mathbf}
\newcommand{\rom}[1]{\uppercase\expandafter{\romannumeral#1}}
\DeclareMathOperator{\id}{id}
\DeclareMathOperator{\Irr}{Irr}
\DeclareMathOperator{\Tr}{Tr}
\DeclareMathOperator{\Ind}{Ind}
\DeclareMathOperator{\CF}{CF}
\DeclareMathOperator{\rank}{rank}
\DeclareMathOperator{\SL}{SL}
\DeclareMathOperator{\GL}{GL}
\DeclareMathOperator{\PGL}{PGL}
\DeclareMathOperator{\PSO}{PSO}
\DeclareMathOperator{\End}{End}
\DeclareMathOperator{\Hom}{Hom}
\DeclareMathOperator{\Aut}{Aut}
\DeclareMathOperator{\IC}{IC}
\DeclareMathOperator{\Uch}{Uch}
\begin{document}

\begin{abstract}
\noindent Let $G$ be a finite Chevalley group. We are concerned with computing the values of the unipotent characters of $G$ by making use of Lusztig's theory of character sheaves. In this framework, one has to find the transformation between several bases for the class functions on $G$. In principle, this has been achieved by Lusztig and Shoji, but the underlying process involves some scalars which are still unknown in many cases. We shall determine these scalars in the specific case where $G$ is one of the groups $\leftidx{}E_6(q)$, $\leftidx{^2}E_6(q)$, and $q$ is a power of the bad prime $p=3$ for $E_6$, by exploiting known facts about the representation theory of the Hecke algebra associated with $G$.
\end{abstract}

\maketitle

\section{Introduction}\label{Intro}

Let $\mf G$ be a simple Chevalley group of type $E_6$ over the algebraic closure $k=\overline{\mathbb F}_p$ of the field with $p$ elements (for a prime $p$). Assume that $\mf G$ is defined over the finite subfield $\mathbb F_q$ of $k$, where $q$ is a power of $p$, so the $\mathbb F_q$-rational points on $\mf G$ constitute the corresponding finite group of Lie type $\mf G(q)=\mf G(\mathbb F_q)$. We are concerned with the problem of computing the values of the ordinary irreducible characters of $\mf G(q)$. To this end, Lusztig's work \cite{Luchars}, \cite{LuCS1}-\cite{LuCS5} is of paramount importance. On the one hand, it can be exploited to directly find the values of irreducible characters of $\mf G(q)$ on unipotent elements provided the characteristic of $k$ is good for $E_6$ (that is, $p\neq2,3$), using the results of \cite{BeSp} and the algorithm in \cite{ShArc} for computing the Green functions. On the other hand and more generally, it basically allows a reformulation of the task: In this setting, one has to determine the transformation between the irreducible characters of $\mf G(q)$ and a further basis of the class functions on $\mf G(q)$, namely the characteristic functions of suitable \enquote{character sheaves} on $\mf G$. More concretely, Lusztig \cite[13.7]{Luchars}, \cite{LuCS5} conjectured that any such characteristic function coincides up to multiplication by a root of unity with an appropriate \enquote{almost character} of $\mf G(q)$, that is, an explicitly known linear combination of the irreducible characters. This conjecture has been proven by Shoji \cite{Sh1}, \cite{Sh2} under the assumption that the center of $\mf G$ is connected. However, the exact values of the scalars relating characteristic functions of character sheaves and almost characters need to be specified. \\
Using this machinery, these scalars are determined for \enquote{cuspidal} unipotent character sheaves in \cite{Gvaluni} when $p=2$. As mentioned in \cite[6.6]{Gvaluni}, the cases where $p$ is a good prime for $E_6$ can be approached similarly (see also \cite{Lucharval}), but the argument there does not work for $p=3$. \\
The purpose of this paper is to specify the scalars relating characteristic functions of cuspidal unipotent character sheaves and the corresponding almost characters of $\mf G(q)$ when $q$ is a power of $p=3$. This will enable us to describe the values of the unipotent characters (as defined in \cite{DL}) at unipotent elements up to a few unknown signs occuring in Lusztig's algorithm for the computation of Green functions. Since the Green functions for the non-twisted group $E_6(q)$ have been computed in \cite{P} via explicit induction of characters from various subgroups, it then actually only remains to identify the missing signs for the non-twisted group $\leftidx{^2}E_6(q)$.

\subsection{Notation}\label{Not}

From now on, except in \Cref{CS}, $\mf G$ denotes a simple Chevalley group of type $E_6$ over $k=\overline{\mathbb F}_3$, an algebraic closure of the field with $3$ elements. Assume that $\mf G$ is defined over $\mathbb F_q$, $q=3^f$ for some $f\geqslant1$, and let $F\colon\mf G\rightarrow\mf G$ be the corresponding Frobenius map. Throughout, we fix an $F$-stable Borel subgroup $\mf B\subseteq\mf G$ and an $F$-stable maximal torus $\mf T\subseteq\mf B$. Let $\mathscr R=(X,R,Y,R^\vee)$ be the root datum attached to $\mf G$ (and $\mf T$), see \cite[1.2, 1.3]{GM}. Here, $X=X(\mf T)=\Hom(\mf T, \mf k^\times)$ is the group of rational characters of $\mf T$, $Y=Y(\mf T)=\Hom(\mf k^\times,\mf T)$ the group of rational cocharacters, $R\subseteq X$ are the roots and $R^\vee=\{\alpha^\vee\mid\alpha\in R\}\subseteq Y$ the coroots of $\mf G$ with respect to $\mf T$. The underlying perfect bilinear pairing is denoted $\langle\;,\;\rangle\colon X\times Y\rightarrow\mathbb Z$. Then $\mf B$ determines the positive roots $R^+\subseteq R$ and in turn the simple roots $\Pi:=\{\alpha_1,\ldots,\alpha_6\}\subseteq R^+$ and simple coroots $\Pi^\vee:=\{\alpha_1^\vee,\ldots,\alpha_6^\vee\}$. We choose the order of $\alpha_1,\ldots, \alpha_6$ in such a way that the Dynkin diagram of $E_6$ is as follows:
\begin{center}
\begin{tikzpicture}
    \draw (-1.25,-0.25) node[anchor=east]  {$E_6$};

    \node[bnode,label=above:$\alpha_1$] 		(1) at (0,0) 	{};
    \node[bnode,label=right:$\alpha_2$] 		(2) at (2,-1) 	{};
    \node[bnode,label=above:$\alpha_3$] 		(3) at (1,0) 	{};
    \node[bnode,label=above:$\alpha_4$] 		(4) at (2,0) 	{};
    \node[bnode,label=above:$\alpha_5$] 		(5) at (3,0) 	{};
    \node[bnode,label=above:$\alpha_6$] 		(6) at (4,0) 	{};
    \node[]						 				(7) at (5,0) 	{};

    \path 	(1) edge[thick, sedge] (3)
          	(3) edge[thick, sedge] (4)
          	(4) edge[thick, sedge] (5)
			(4)	edge[thick, sedge] (2)
          	(5) edge[thick, sedge] (6);
\end{tikzpicture}
\end{center}
Let $C:=(\langle\alpha_j,\alpha_i^\vee\rangle)_{1\leqslant i,j\leqslant6}$ be the corresponding Cartan matrix and $W:=N_{\mf G}(\mf T)/{\mf T}$ be the Weyl group of $\mf G$. The conjugation action of $N_{\mf G}(\mf T)$ on $\mf T$ induces an action of $W$ on $X$ which allows us to identify $W$ with a subgroup of $\Aut(X)$. Then $W$ can be viewed as a Coxeter group of type $E_6$ with Coxeter generators $S=\{s_1,\ldots, s_6\}$, where $s_i=w_{\alpha_i}\in\Aut(X)$ ($1\leqslant i\leqslant6$) are given by
\[s_i(\lambda)=\lambda-\langle\lambda,\alpha_i^\vee\rangle\alpha_i\qquad\text{for}\quad\lambda\in X.\]
Furthermore, let $\mf U:=R_u(\mf B)$ be the unipotent radical of $\mf B$. Then $\mf B$ is the semidirect product of $\mf U$ and $\mf T$ (with $\mf U$ being normal in $\mf B$). Now $\mf T$ is $F$-stable, hence so is $N_{\mf G}(\mf T)$, and $F$ induces an automorphism on $W$ which we denote by $\sigma\colon W\rightarrow W$. On the other hand, $F$ also induces a group homomorphism $\varphi\colon X\rightarrow X$, $\lambda\mapsto \lambda\circ F|_{\mf T}$, and this defines a $p$-isogeny of root data as in \cite[1.2.9]{GM}: There is a permutation $R\rightarrow R$, $\alpha\mapsto\alpha^\dagger$, such that $\varphi(\alpha^\dagger)=q\alpha$ for all $\alpha\in R$, see \cite[1.4.26]{GM}. Here, the assignment $\alpha\mapsto\alpha^\dagger$ restricts to a graph automorphism of the Dynkin diagram, so there are two possible cases: Either $\alpha\mapsto\alpha^\dagger$ is the identity (then $\mf G^F$ is the untwisted group $E_6(q)$ and $\sigma=\id_W$), or else it is a map of order $2$ (then $\mf G^F$ is the twisted group $\leftidx{^2}E_6(q)$ and $\sigma\colon W\rightarrow W$ is the inner automorphism given by conjugation with the longest element $w_0$ of $W$). \\
We will be concerned with characters of the finite group of Lie type $\mf G^F$ in characteristic $0$. As usual in the ordinary representation theory of finite groups of Lie type, we consider representations and characters over $\overline{\mathbb Q}_\ell$, an algebraic closure of the field of $\ell$-adic numbers, for a fixed prime $\ell$ different from $p$. Thus given a finite group $\Gamma$, let $\CF(\Gamma)$ be the set of class functions $\Gamma\rightarrow\overline{\mathbb Q}_\ell$ and let $\langle f, f'\rangle_{\Gamma}:=|\Gamma|^{-1}\sum_{g\in\Gamma}f(g)\overline{f'(g)}$ ($f, f'\in\CF(\Gamma)$) be the standard scalar product on $\CF(\Gamma)$, where bar denotes a field automorphism of $\overline{\mathbb Q}_\ell$ which maps roots of unity to their inverses. We denote by $\Irr(\Gamma)\subseteq\CF(\Gamma)$ the subset of irreducible characters of $\Gamma$, which form an orthonormal basis of $\CF(\Gamma)$ with respect to this scalar product. Now let $\Gamma=\mf G^F$ and consider the subset $\Uch(\mf G^F)\subseteq\Irr(\mf G^F)$ of unipotent characters, that is, those $\rho\in\Irr(\mf G^F)$ which satisfy $\langle\rho,R_w\rangle\neq0$ for some $w\in W$. Here, $R_w$ is the virtual character defined by Deligne and Lusztig in \cite{DL}. Our aim is to determine the values of the unipotent characters at unipotent elements of $\mf G^F$. Note that, in terms of the group $\mf G^F$, it is immaterial whether we start with the adjoint group $\mf G=\mf G_\mathrm{ad}$ or the simply connected group $\mf G=\mf G_\mathrm{sc}$ of type $E_6$. Indeed, since $k$ has characteristic $3$ and since the fundamental group $\Lambda(C)$ of $C$ is isomorphic to $\mathbb Z_3$, the group $\Hom(\Lambda(C),k^\times)$ is trivial, so the center of $\mf G_{\mathrm{sc}}$ is likewise (\cite[1.5.2]{GM}). Hence, we obtain an isomorphism between $\mf G^F_{\mathrm{sc}}$ and $\mf G^F_{\mathrm{ad}}$, see \cite[1.5.12]{GM}. For our purposes, we can thus assume without loss of generality that $\mf G=\mf G_\mathrm{sc}$ is the semisimple, simply connected group of type $E_6$ over $k$.

\section{Lusztig's classification of unipotent characters}

According to \cite[4.23]{Luchars}, $\Uch(\mf G^F)$ can be classified in terms of the following data, which only depend on $W$ and the automorphism $\sigma\colon W\rightarrow W$, and not on the power $q$ of $p$. Denote by $\overline X(W,\sigma)$ a parameter set for the unipotent characters of $\mf G^F$:
\begin{equation}\label{ParUch}
\Uch(\mf G^F)\leftrightarrow\overline X(W,\sigma),\quad\rho\leftrightarrow\overline x_\rho.
\end{equation}
$\overline X(W,\sigma)$ is equipped with a pairing $\{\;,\;\}\colon\overline X(W,\sigma)\times\overline X(W,\sigma)\rightarrow\overline{\mathbb Q}_\ell$. Let $\Irr(W)^\sigma$ be the set of all $\phi\in\Irr(W)$ which satisfy $\phi\circ\sigma=\phi$. Since $\sigma$ is an inner automorphism of $W$, this condition is always true, so we have $\Irr(W)^\sigma=\Irr(W)$ and we can henceforth drop the superscript $\sigma$. Let
\begin{equation}\label{IrrWembed}
\Irr(W)\hookrightarrow\overline X(W,\sigma),\quad \phi\mapsto x_\phi
\end{equation}
be the embedding defined in \cite[(4.21.3)]{Luchars}. Let $d\leqslant2$ be the order of $\sigma\in\Aut(W)$ and consider the semidirect product $W_d:=\langle\sigma\rangle\ltimes W$, where $\sigma\cdot w\cdot\sigma^{-1}=\sigma(w)$ for $w\in W$. Any irreducible representation $\Theta\colon W\rightarrow\GL_n(\overline{\mathbb Q}_\ell)$ can thus be extended in $d$ different ways to an irreducible representation $\tilde\Theta\colon W_d\rightarrow\GL_n(\overline{\mathbb Q}_\ell)$. If $d=2$ (that is, we are in the case of $\leftidx{^2}E_6(q)$ and $\sigma$ is given by conjugation with $w_0$), we have $\tilde\Theta(\sigma)=\delta\Theta(w_0)$ where $\delta\in\{\pm1\}$, and this determines the two extensions of $\Theta$. Let $\phi\in\Irr(W)$ be the character afforded by $\Theta$. Then there is a corresponding $\sigma$-class function $\tilde\phi[\delta]\colon W\rightarrow\overline{\mathbb Q}_\ell$, defined by $\tilde\phi[\delta](w)=\delta\phi(ww_0)$ for any $w\in W$. We will tacitly identify $\tilde\phi[\delta]$ with the actual extension of $\phi$, that is, the irreducible character of $W_d$ afforded by $\tilde\Theta$. Let
\[R_{\tilde\phi[\delta]}:=\frac 1{|W|}\sum_{w\in W}\tilde\phi[\delta](w)R_w \quad(\phi\in\Irr(W),\; d=2).\]
Now, Lusztig defines another set $X(W,\sigma)$ such that the pairing on $\overline X(W,\sigma)$ induces a pairing $\{\;,\;\}\colon\overline X(W,\sigma)\times X(W,\sigma)\rightarrow\overline{\mathbb Q}_\ell$ and the group of all roots of unity in $\overline{\mathbb Q}_\ell^\times$ acts freely on $X(W,\sigma)$. The set of orbits under this action is in bijective correspondence with $\overline X(W,\sigma)$. For any $x\in X(W,\sigma)$, there is a corresponding unipotent \enquote{almost character} $R_x$, defined by
\[R_x:=\sum_{\rho\in\Uch(\mf G^F)}\{\overline x_\rho,x\}\Delta(\overline x_\rho)\rho,\]
where $\Delta(\overline x_\rho)\in\{\pm1\}$ is a certain sign attached to $\rho\in\Uch(\mf G^F)$, see \cite[4.21]{Luchars}. Up to multiplication by a root of unity, $R_x$ only depends on the orbit of $x\in X(W,\sigma)$. By the description in \cite[4.19]{Luchars}, it suffices for our purposes to consider a finite subset of $X(W,\sigma)$ which can be identified with $\overline X(W,\sigma)\times M_d$, where $M_d$ denotes the group of all $d$-th roots of unity in $\overline{\mathbb Q}_\ell^\times$. With these notions, the scalar products above are related as follows. We have $\{x,(y,a)\}=a^{-1}\{x,y\}$ for any $x,y\in\overline X(W,\sigma)$, $a\in M_d$. In particular, $R_{(x,-1)}=-R_{(x,1)}$ for any $x\in\overline X(W,\sigma)$ (in the case of $d=2$). The above action of all roots of unity on $X(W,\sigma)$ restricts to an action of $M_d$ on $\overline X(W,\sigma)\times M_d$ which is given by (left) multiplication on the second factor. Furthermore, the embedding $\Irr(W)\hookrightarrow\overline X(W,\sigma)$ induces an embedding $\Irr(W_d)\hookrightarrow\overline X(W,\sigma)\times M_d$ such that an extension of $\phi\in\Irr(W)$ is mapped to $(x_\phi,a)$, $a\in M_d$. More precisely, given any $\phi\in\Irr(W)$, let us from now on denote by $\tilde\phi$ the \enquote{preferred} extension defined in \cite[17.2]{LuCS4}. Then $\tilde\phi$ is mapped to $(x_\phi,1)$ under the embedding $\Irr(W_d)\hookrightarrow\overline X(W,\sigma)\times M_d$, see again \cite[4.19]{Luchars}. Identifying $\overline X(W,\sigma)$ with $\{(x,1)\mid x\in\overline X(W,\sigma)\}\subseteq\overline X(W,\sigma)\times M_d$, we have $R_{\tilde\phi}=R_{x_\phi}$ for any $\phi\in\Irr(W)$ by \cite[4.24]{Luchars}. Since the \enquote{Fourier matrix} $\Upsilon:=\{x,x'\}_{x,x'\in\overline X(W,\sigma)}$ is hermitian and $\Upsilon^2$ is the identity matrix (see \cite[\S 4]{LuUniE8}), we obtain
\[\langle R_x,R_{x'}\rangle_{G^F}=\begin{cases}1&\text{ if }x=x',\\0&\text{ if }x\neq x'\end{cases}\quad(\text{for }x,x'\in\overline X(W,\sigma)).\]
It follows that
\begin{equation}\label{Fourier}
\rho=\Delta(\overline x_\rho)\sum_{x\in\overline X(W,\sigma)}{\{\overline x_\rho,x\}}R_x\quad\text{for }\rho\in\Uch(\mf G^F)
\end{equation}
(see \cite[4.25]{Luchars}, note that the numbers $\{\overline x_\rho,x\}$ are all rational in the case of $E_6$). In particular, knowing the values of the unipotent characters of $\mf G^F$ is the same as knowing the values of the $R_x$. Now, Lusztig's fundamental algorithm in \cite[24.4]{LuCS5} yields expressions for the $R_{\tilde\phi}$ ($\phi\in\Irr(W)$) as linear combinations of certain class functions $Y_{\tilde\psi}$ ($\psi\in\Irr(W)$). This is implemented in {\sffamily {CHEVIE}} (\cite{MiChv}). Up to a few signs, the values of the $Y_{\tilde\psi}$ can be computed. Hence, once these signs are determined, the values of the $R_{\tilde\phi}$ ($\phi\in\Irr(W)$) at unipotent elements of $\mf G^F$ will be known. At least for the non-twisted group $E_6(q)$, the functions $R_{\tilde\phi}|_{\mf G^F_{\mathrm{uni}}}$ were computed explicitly in \cite{P} by inducing characters from various smaller subgroups. \\
Now, we have $|\overline X(W,\sigma)|=|\Uch(\mf G^F)|=30$ while $|\Irr(W)|=25$, see \cite[p.~480]{C}. In order to solve the problem of computing the values of the $5$ almost characters which do not arise from irreducible characters of $W$, we make use of Lusztig's theory of character sheaves.

\section{Character sheaves}\label{CS}

Assume in this section that $\mf G$ is an arbitrary connected reductive group over $\overline{\mathbb F}_p$ ($p>0$ any prime number), defined over $\mathbb F_q$ ($q$ any power of $p$) with corresponding Frobenius map $F\colon\mf G\rightarrow\mf G$. In \cite[13.7]{Luchars}, Lusztig conjectured that there is a geometric analogue to the irreducible characters of $\mf G^F$, giving rise to a further basis of the class functions on $\mf G^F$ which essentially coincides with the basis consisting of almost characters. In the sequel, he developed the theory of character sheaves \cite{LuCS1}-\cite{LuCS5}, which was completed quite recently \cite{LuclCS}. The results of Shoji \cite{Sh1}, \cite{Sh2} give an affirmative answer to Lusztig's conjecture, at least if the center of $\mf G$ is connected. We begin by very briefly introducing some notions of this theory (only those which are relevant for our purposes), for details see \cite[\S 7]{Gsurvey}, \cite[2.4-2.6]{Gvaluni} and of course the main references \cite{LuCS1}-\cite{LuCS5}, \cite{Sh1}, \cite{Sh2}. \\
Let $\hat{\mf G}$ be the set of isomorphism classes of character sheaves on $\mf G$. These are certain irreducible perverse sheaves in the bounded derived category $\mathscr D\mf G$ of constructible $\overline{\mathbb Q}_\ell$-sheaves on $\mf G$ (in the sense of \cite{BBD}), which are equivariant for the conjugation action of $\mf G$ on itself. (For the precise definition of $\hat{\mf G}$, see \cite[2.10]{LuCS1}.) Now, if $A\in\mathscr D\mf G$, let $F^\ast A\in\mathscr D\mf G$ be the inverse image of $A$ under the Frobenius map $F$. Suppose that $F^\ast A$ is isomorphic to $A$ and choose an isomorphism $\varphi\colon F^\ast A\xrightarrow{\sim} A$. Then $\varphi$ induces linear maps $\varphi_{i,g}\colon \mathscr H_g^i(A)\rightarrow\mathscr H_g^i(A)$ for $i\in\mathbb Z$ and $g\in\mf G^F$, where $\mathscr H_g^i(A)$ is the stalk at $g$ of the $i$th cohomology sheaf of $A$. In turn, Lusztig \cite[8.4]{LuCS2} defines the characteristic function $\chi_{A,\varphi}\in\CF(\mf G^F)$ associated with $A$ (and $\varphi$):
\[\chi_{A,\varphi}(g):=\sum_{i\in\mathbb Z}(-1)^i\mathrm{Trace}(\varphi_{i,g},\mathscr H_g^i(A))\quad\text{for }g\in\mf G^F.\]
This is well-defined since only finitely many of the $\mathscr H_g^i(A)$ ($i\in\mathbb Z$) are non-zero. Note however that $\varphi\colon F^\ast A\xrightarrow{\sim}A$ is only unique up to a non-zero scalar. Denote by ${\hat{\mf G}}^F\subseteq\hat{\mf G}$ the $F$-stable character sheaves on $\mf G$, i.\,e. those $A\in\hat{\mf G}$ satisfying $F^\ast A\cong A$. For $A\in{\hat{\mf G}}^F$, an isomorphism $\varphi_A\colon F^\ast A\xrightarrow{\sim}A$ can be chosen in such a way that the values of the characteristic functions $\chi_{A,\varphi_A}$ are cyclotomic integers and
\[\langle \chi_{A,\varphi_A},\chi_{A',\varphi_{A'}}\rangle=\begin{cases}1&\text{ if }A=A',\\0&\text{ if }A\neq A'\end{cases}\qquad\text{for }A, A'\in{\hat{\mf G}}^F\]
(see \cite[25.6, 25.7]{LuCS5}). The required properties for the $\varphi_A$ ($A\in{\hat{\mf G}}^F$) according to \cite[25.1]{LuCS5}, \cite[13.8]{LuCS3}, determine $\varphi_A$ up to multiplication by a root of unity. Whenever $A\in{\hat{\mf G}}^F$, we will tacitly assume that an isomorphism $\varphi_A\colon F^\ast A\xrightarrow{\sim}A$ as above has been chosen, and we just write $\chi_A$ instead of $\chi_{A,\varphi_A}$. \\
Furthermore, let ${\hat{\mf G}}^\circ\subseteq\hat{\mf G}$ be the set of \enquote{cuspidal character sheaves} on $\mf G$ defined in \cite[3.10]{LuCS1}. By \cite{LuCS1}, \cite[\S 4]{Shcompuni} and since the results in \cite{LuCS5} are known to hold in complete generality (\cite{LuclCS}), we obtain a characterisation of $F$-stable cuspidal character sheaves which highlights the analogy to cuspidal characters of $\mf G^F$. Recall \cite[7.2]{LuIrrCl} that a regular subgroup $\mf L$ of $\mf G$ is an $F$-stable subgroup which is the Levi subgroup of some (not necessarily $F$-stable) parabolic subgroup $\mf P$ of $\mf G$, and $R_{\mf L}^{\mf G}$ is \enquote{twisted induction}, defined in \cite{LuUninumber}. Then we have
\[{\hat{\mf G}}^{\circ F}=\bigl\{A\in{\hat{\mf G}}^F\,\big\vert\, \langle \chi_A, R_{\mf L}^{\mf G}(f)\rangle_{\mf G^F}=0 \text{ for any }\mf L\subsetneq\mf G \text{ regular, }f\in\CF(\mf L^F)\bigr\}.\]
Furthermore, the class functions on $\mf G^F$ can be described via twisted induction and cuspidal character sheaves of regular subgroups of $\mf G$, see also \cite[7.11]{Gsurvey}:
\begin{equation}\label{DL-cusp}
\CF(\mf G^F)=\langle R_{\mf L}^{\mf G}(\chi_{A_0})\mid\mf L\subseteq\mf G\text{ regular and }A_0\in{\hat{\mf L}}^{\circ F}\rangle_{\overline{\mathbb Q}_\ell}.
\end{equation}
More precisely, it follows from \cite[(10.4.5), (10.6.1)]{LuCS5}, \cite[\S 4]{Shcompuni}, that each characteristic function of a cuspidal character sheaf of $\mf G$ is a linear combination of various $R_{\mf L}^{\mf G}(\chi_{A_0})$ such that every $\mf L$ occuring in the decomposition has the same Cartan type. \\
Finally, given $w\in W$, let $K_w:=K_w^{{\mathscr L}_0}\in\mathscr D\mf G$ be as defined in \cite[2.4]{LuCS1} with respect to the constant $\overline{\mathbb Q}_\ell$-local system $\mathscr L_0=\overline{\mathbb Q}_\ell$ on $\mf T$. An element of $\hat{\mf G}$ is called a unipotent character sheaf if it is a constituent of a perverse cohomology sheaf $\leftidx{^p}H^i(K_w)$ for some $i\in\mathbb Z$, $w\in W$. Denote by ${\hat{\mf G}}^\mathrm{un}$ the subset of $\hat{\mf G}$ consisting of the (isomorphism classes of) unipotent character sheaves on $\mf G$. Now, the set $\overline X(W,\sigma)$ in \eqref{ParUch} also serves as a parameter set for ${\hat{\mf G}}^\mathrm{un}$:
\begin{equation}\label{ParUchSh}
\overline X(W,\sigma)\leftrightarrow{\hat{\mf G}}^\mathrm{un},\quad x\leftrightarrow A_x,
\end{equation}
subject to a property involving the Fourier matrix (\cite[23.1]{LuCS5}). With these notions we can formulate the following theorem of Shoji which verifies Lusztig's conjecture under the assumption that $\mf G$ has connected center. As mentioned in \cite[2.7]{Gvaluni}, this holds without any conditions on $p$, $q$, since the cleanness of cuspidal character sheaves is established in full generality (\cite{LuclCS}).
\begin{Thm}[Shoji {\cite[3.2, 4.1]{Sh2}}]\label{Shoji}
Let $p$ be a prime, $q$ a power of $p$, $\mf G$ a connected reductive group over $\overline{\mathbb F}_p$, defined over $\mathbb F_q$ with corresponding Frobenius map $F\colon\mf G\rightarrow\mf G$. Assume that $\mf Z(\mf G)$ is connected and $\mf G/{\mf Z(\mf G)}$ is simple. Then ${\hat{\mf G}}^\mathrm{un}\subseteq{\hat{\mf G}}^F$ and for any $x\in\overline X(W,\sigma)$, $R_x$ and $\chi_{A_x}$ coincide up to a non-zero scalar.
\end{Thm}

\section{Character values}

The notation and assumptions are as in \ref{Not}. In particular, $\mf G$ has trivial center and we can apply \Cref{Shoji}. So there are scalars $\xi_x\in\overline{\mathbb Q}_\ell$ such that
\begin{equation}\label{scalars}
R_x=\xi_x\chi_{A_x}\quad\text{for }x\in\overline X(W,\sigma).
\end{equation}
Since $\langle R_x,R_x\rangle=1=\langle\chi_A,\chi_A\rangle$ for any $x\in\overline X(W,\sigma)$, $A\in{\hat{\mf G}}^\mathrm{un}$, we know that $\xi_x\overline\xi_x=1$ for any $x\in\overline X(W,\sigma)$.
By \cite[20.3]{LuCS4} and \cite[4.6]{Sh2} there are two cuspidal character sheaves $A_1, A_2$ for $\mf G$, and both of them lie in ${\hat{\mf G}}^\mathrm{un}$. Their support is the unipotent variety $\mf G_{\mathrm{uni}}$ of $\mf G$ consisting of all unipotent elements in $\mf G$. This variety is the (Zariski-)closure of the regular unipotent conjugacy class $\mathscr O_{\mathrm{reg}}$, which is the unique class of all $x\in\mf G$ with the property $\dim C_{\mf G}(x)=\rank\mf G=6$. In particular $\mathscr O_{\mathrm{reg}}$ is $F$-stable. Using \cite[3.12]{LuCS1}, we conclude that there exist irreducible, $\mf G$-equivariant $\overline{\mathbb Q}_\ell$-local systems $\mathscr E_1$, $\mathscr E_2$ on $\mathscr O_\mathrm{reg}$ such that $F^\ast\mathscr E_i\cong\mathscr E_i$ and $A_i=\IC(\mf G_{\mathrm{uni}},\mathscr E_i)[\dim\mathscr O_\mathrm{reg}]$ for $i=1,2$. (\enquote{IC} stands for the intersection cohomology complex due to Deligne-Goresky-MacPherson (\cite{GMP}, \cite{BBD}), see \cite{LuIC}.) Fix an element $u_0\in\mathscr O_\mathrm{reg}^F$ and set $A(u_{0}):=C_{\mf G}(u_{0})/C_{\mf G}^\circ(u_{0})$. This is a cyclic group of order $3$ generated by the image of $u_{0}$ (see \cite[\S 4]{PW} and \cite[14.15, 14.18]{DM}). Thus the automorphism of $A(u_{0})$ induced by $F$ is the identity and the elements of $A(u_0)$ correspond to the $\mf G^F$-conjugacy classes contained in $\mathscr O_\mathrm{reg}^F$ (see, for instance, \cite[4.3.6]{Gintro}). In particular, there are $3$ such classes. On the other hand, as described in \cite[4.6]{Sh2}, the $A_i$ ($i=1,2$) correspond to the two non-trivial linear characters of $A(u_{0})$. Hence, by the construction in \cite[19.7]{LuCSDC4}, $A_1, A_2$ give rise to the following two characteristic functions $\chi_1, \chi_2$, where $\theta$ is a primitive third root of unity, $g$ an element of $\mf G^F$ and $u_{0}$ (chosen as above), $u_0'$, $u_0''$ are representatives of the $\mf G^F$-classes inside $\mathscr O_\mathrm{reg}^F$:
\begin{center}
\begin{tabular}{|c|c|c|c|c|}
\hline
& $g\notin\mathscr O_\mathrm{reg}^F$ & $g=u_{0}$ & $g=u_0'$ & $g=u_0''$ \\
\hline
$\chi_1(g)$ & 0 & $q^3$ & $q^3\theta$ & $q^3\theta^2$ \\
\hline
$\chi_2(g)$ & 0 & $q^3$ & $q^3\theta^2$ & $q^3\theta$  \\
\hline
\end{tabular}
\end{center}
(The factor $q^3=q^{{(\dim\mf G-\dim\mathscr O_\mathrm{reg}})/2}$ ensures that the $\chi_i$ have norm 1.) Let $x_1,x_2$ be the corresponding elements of $\overline X(W,\sigma)$ according to \eqref{ParUchSh}, so that $\chi_i=\chi_{A_{x_i}}$, $i=1,2$. The proof of \cite[20.3]{LuCS4} shows that the unipotent characters labelled by $x_1, x_2$ are indeed the cuspidal unipotent characters $E_6[\theta], E_6[\theta^2]$ in the non-twisted case, respectively $\leftidx{^2}E_6[\theta], \leftidx{^2}E_6[\theta^2]$ in the twisted case (where we use the notation in Carter's tables \cite[pp.~480-481]{C} and $\theta$ is as above). Now suppose $x\in\overline X(W,\sigma)\backslash\{x_1,x_2\}$ does not arise from an irreducible character of $W$ via the embedding $\Irr(W)\hookrightarrow\overline X(W,\sigma)$ \eqref{IrrWembed}. We want to show that $R_x$ (or, equivalently, $\chi_{A_x}$) vanishes on all unipotent elements of $\mf G^F$.
\begin{Lm}\label{zerouni}
Let $x\in\overline X(W,\sigma)\backslash\{x_1,x_2\}$ be an element which is not in the image of the map $\Irr(W)\hookrightarrow\overline X(W,\sigma)$ in \eqref{IrrWembed}. Then the characteristic function $\chi_x$ of the character sheaf $A_x$ is a linear combination of various $R_{\mf L}^{\mf G}(\chi_{A_0})$ where $\mf L\subseteq\mf G$ is a regular subgroup of Cartan type $D_4$ and $A_0$ is a cuspidal character sheaf on $\mf L$.
\end{Lm}
\begin{proof}
According to \eqref{DL-cusp} and since we assumed $A_x$ to be non-cuspidal, $\chi_x$ can be written as a linear combination of suitable $R_{\mf L}^{\mf G}(\chi_{A_0})$ where the $\mf L\subsetneq\mf G$ are proper regular subgroups, all of the same Cartan type, and $A_0\in\hat{\mf L}^{\circ F}$. First we show that, given any proper regular subgroup $\mf L$ of $\mf G$ which is neither a torus nor is of Cartan type $D_4$, there are no cuspidal character sheaves for $\mf L$. Using \cite[17.10]{LuCS4}, we can replace $\mf L$ by the semisimple group $\mf L_\mathrm{ss}:={\mf L}/{\mf Z(\mf L)^\circ}$. Let $\mf L_{\mathrm{ss}}^1,\ldots, \mf L_{\mathrm{ss}}^r$ ($r\geqslant1$) be the simple factors of $\mf L_\mathrm{ss}$, so that the product map defines an isogeny $\mf L_{\mathrm{ss}}^1\times\ldots\times\mf L_{\mathrm{ss}}^r\rightarrow\mf L_{\mathrm{ss}}$. Note that, if $r\geqslant2$, the only possible Cartan types for the $\mf L_{\mathrm{ss}}^i$ are $A_1, A_2, A_3, A_4$. Thus in order to prove our claim, we may assume that $\mf L_{\mathrm{ss}}$ is simple (in view of \cite[17.11, 17.16]{LuCS4}). \\
There are no cuspidal character sheaves for $\mf L$ of type $D_5$, see the proof of \cite[19.3]{LuCS4}. We claim that there are no cuspidal unipotent character sheaves for $\mf L$ of type $A_n$ with $n\geqslant1$ as well. Indeed, $\mf Z(\mf G)=\mf Z(\mf G)^\circ$ implies $\mf Z(\mf L)=\mf Z(\mf L)^\circ$ (\cite[13.14]{DM}) and, hence, $\mf Z(\mf L_\mathrm{ss})=\{1\}$. Now $\PGL_{n+1}(k)={\GL_{n+1}(k)}/{\mf Z(\GL_{n+1}(k))}$ is the semisimple adjoint group of type $A_n$, and the kernel of the corresponding central isogeny $f\colon\mf L_{\mathrm{ss}}\rightarrow\PGL_{n+1}(k)$ must be trivial. Applying the argument in \cite[2.10]{LuIC} to $f$, we are reduced to the case where $\mf L_{\mathrm{ss}}=\PGL_{n+1}(k)$, $n\geqslant1$. Hence, by \cite[17.10]{LuCS4}, it remains to note that there are no cuspidal character sheaves for $\GL_{n+1}(k)$. This is clear from the introduction in \cite{LuIC} since centralisers of invertible matrices are always connected, so the group $A(g)$ of components of the centraliser of any $g\in\GL_{n+1}(k)$ is trivial. \\
So $\chi_x$ can be written as a linear combination of $R_{\mf L}^{\mf G}(\chi_{A_0})$ where either each $\mf L$ is of type $D_4$ or each $\mf L$ is an $F$-stable maximal torus of $\mf G$. Assume, if possible, that we are in the latter case, so $\chi_x$ decomposes as a linear combination of the virtual Deligne-Lusztig characters $R_{\mf T_w}^{\mf G}(\theta)$ where $w\in W$, $\mf T_w$ is a torus of type $w$ with respect to $\mf T$, and $\theta\in\Irr(\mf T_w^F)$.
Since the unipotent characters form a single geometric conjugacy class (see \cite[\S12.1]{C}), we have $\langle\rho,R_{\mf T_w}^{\mf G}(\theta)\rangle_{\mf G^F}=0$ for any $\rho\in\Uch(\mf G^F)$, $w\in W$ and $1\neq\theta\in\Irr(\mf T^F)$. It follows that $\langle R_x,R_{\mf T_w}^{\mf G}(\theta)\rangle_{\mf G^F}=0$ for any $w\in W$, $1\neq\theta\in\Irr(\mf T^F)$ and, hence, $\chi_x=\xi_x^{-1}R_x$ is in fact a linear combination of various $R_w=R_{\mf T_w}^{\mf G}(1)$ ($w\in W$). However, the $R_w$ correspond to the $\mf G^F$-conjugacy classes of $F$-stable maximal tori in $\mf G$, which in turn correspond to the $\sigma$-conjugacy classes in $W$ (\cite[\S4.3]{Gintro}). So there are $|\Irr(W)^\sigma|=|\Irr(W)|$ many different $R_w$ which contradicts the orthogonality of the class functions $R_{\tilde\phi}$ ($\phi\in\Irr(W)$) along with $\chi_x$. The lemma is proved.
\end{proof}
\begin{Prop}
If $x$, $\chi_x$ are as in \Cref{zerouni}, then $\chi_x(u)=0$ for any $u\in\mf G_{\mathrm{uni}}^F$.
\end{Prop}
\begin{proof}
Let $\mf L$ be a regular subgroup of $\mf G$ of Cartan type $D_4$, $\mf L_{\mathrm{ss}}=\mf L/{\mf Z(\mf L)}^\circ$ and $\pi\colon\mf L\rightarrow\mf L_{\mathrm{ss}}$ the natural map. $\mf L_{\mathrm{ss}}$ has trivial center (see the proof of \Cref{zerouni}) and thus is isomorphic to $\PSO_8(k)$, the adjoint semisimple group of type $D_4$. By \cite[19.3]{LuCS4}, all the cuspidal character sheaves of $\mf L_{\mathrm{ss}}$ have the same support, namely the closure of the conjugacy class $C\subseteq\mf L_{\mathrm{ss}}$ of $s_0u_0$, where $s_0$ is a semisimple element such that $C_{\mf L}^\circ(s_0)$ is isogenous to $\SL_2(k)^4$ and $u_0$ is a regular unipotent element in $C_{\mf L}^\circ(s_0)$. In particular, $s_0\neq1$. Now, the cuspidal character sheaves of $\mf L$ are supported by the closure of $\pi^{-1}(C)$ (see \cite[17.10]{LuCS4}, \cite[2.10]{LuIC}), and $\pi^{-1}(C)$ does not contain any unipotent elements. The cleanness of cuspidal character sheaves (see \cite[23.1]{LuCS5}) implies $\chi_{A_0}(v)=0$ for $v\in\mf L_\mathrm{uni}^F$ whenever $A_0$ is an $F$-stable cuspidal character sheaf of $\mf L$. Given any $\chi\in\Irr(\mf L^F)$ and $g\in\mf G^F$ with Jordan decomposition $g=su=us$ ($s$ semisimple, $u$ unipotent), we have
\[\left(R_{\mf L}^{\mf G}\chi\right)(g)=\frac1{\bigl|\mf L^F\bigr|\bigl|C_{\mf G}^\circ(s)^F\bigl|}\sum_{h\in\mf G^F:\,s\in\leftidx{^h}{\mf L}{}}\left|C_{\leftidx{^h}{\mf L}{}}^\circ(s)^F\right|\sum_{v\in C_{\leftidx{^h}{\mf L}{}}^\circ(s)_\mathrm{uni}^F}Q_{C_{\leftidx{^h}{\mf L}{}}^\circ(s)}^{C_{\mf G}^\circ(s)}(u,v^{-1})\cdot\leftidx{^h}\chi{}(sv),\]
where $Q_{C_{\leftidx{^h}{\mf L}{}}^\circ(s)}^{C_{\mf G}^\circ(s)}$ denotes the two-variable Green function, see \cite[12.2]{DM}. By linearity, we can replace $\chi$ by $\chi_{A_0}$ in the above formula and we get $\left(R_{\mf L}^{\mf G}(\chi_{A_0})\right)(u)=0$ for all $u\in\mf G_\mathrm{uni}^F$. \Cref{zerouni} yields the result.
\end{proof}
Using \eqref{Fourier} and \eqref{scalars}, we thus obtain
\begin{equation}\label{Fourierrest}
\Delta(\overline x_\rho)\cdot\rho|_{{\mf G}^F_\mathrm{uni}}=\sum_{\phi\in\Irr(W)}\{\overline x_\rho, x_\phi\}R_{\tilde\phi}|_{{\mf G}^F_\mathrm{uni}}+\sum_{i=1}^2\{\overline x_\rho,x_i\}\xi_{x_i}\chi_i|_{{\mf G}^F_\mathrm{uni}}
\end{equation}
for $\rho\in\Uch(\mf G^F)$. On the other hand, denote by $\rho_x\in\Uch(\mf G^F)$ the unipotent character corresponding to $x\in\overline X(W,\sigma)$ in \eqref{ParUch}. Then the definition of $R_{x_1}$ reads
\[R_{x_1}=\sum_{x\in\overline X(W,\sigma)}\{x,x_1\}\Delta(x)\rho_x=\frac23\rho_{x_1}-\frac13\rho_{x_2}+\sum_{x\in\overline X(W,\sigma)\backslash\{x_1,x_2\}}\{x,x_1\}\rho_x.\]
As mentioned earlier, we have $\rho_{x_1}=E_6[\theta]$, $\rho_{x_2}=E_6[\theta^2]$ in the non-twisted case and $\rho_{x_1}=\leftidx{^2}E_6[\theta]$, $\rho_{x_2}=\leftidx{^2}E_6[\theta^2]$ in the twisted case. It follows from \cite{GSchur} that $\overline\rho_{x_1}=\rho_{x_2}$ since in either case $\rho_{x_1}$, $\rho_{x_2}$ are the (only) cuspidal unipotent characters with non-trivial character field. We get
\begin{equation}\label{R1bar}
\overline R_{x_1}=\frac23\rho_{x_2}-\frac13\rho_{x_1}+\sum_{x\in\overline X(W)\backslash\{x_1,x_2\}}\{x,x_1\}\overline{\rho}_x.
\end{equation}
Now recall that we started with an arbitrary $u_0\in\mathscr O_\mathrm{reg}^F$. We will now make a definite choice for $u_0\in\mf U\cap\mathscr O_\mathrm{reg}^F$, depending on $\sigma\colon W\rightarrow W$. Denote by $u_i=u_{\alpha_i}$ ($1\leqslant i\leqslant 6$) the closed embedding $\mf k^+\rightarrow\mf G$ whose image is the root subgroup $\mf U_{\alpha_i}\subseteq\mf U$. We set
\begin{equation}\label{test}
\begin{aligned}
u_{0}&:=u_1(1)\cdot u_2(1)\cdot u_3(1)\cdot u_4(1)\cdot u_5(1)\cdot u_6(1)\in\mf U \quad\text{if } \sigma=\id_W, \\
u_{0}&:=u_1(1)\cdot u_6(1)\cdot u_3(1)\cdot u_5(1)\cdot u_2(1)\cdot u_4(1)\in\mf U \quad\text{if } \sigma\neq\id_W.
\end{aligned}
\end{equation}
Then $u_0\in\mathscr O_\mathrm{reg}^F$ in either case. 
\begin{Lm}
With the choices in \eqref{test}, $u_0$ is conjugate to $u_0^{-1}$ in $\mf G^F$.
\end{Lm}
\begin{proof}
There is an isomorphism of abelian groups
\[k^\times\otimes_{\mathbb Z}Y(\mf T)\xrightarrow{\sim}\mf T,\quad\xi\otimes\nu\mapsto\nu(\xi).\]
Since $\mf G$ is of simply connected type, we have $Y(\mf T)=\mathbb ZR^\vee=\mathbb Z\Pi^\vee$, so every element of $\mf T$ has the form $\alpha_1^\vee(\xi_1)\cdot\ldots\cdot\alpha_6^\vee(\xi_6)$ for some uniquely determined $\xi_1,\ldots, \xi_6\in k^\times$. In order to get the correct coefficients, first conjugate $u_0$ by $t:=\alpha_2^\vee(-1)\alpha_3^\vee(-1)\alpha_4^\vee(-1)\alpha_5^\vee(-1)\in\mf T^F$. Next, in analogy to the conjugacy of Coxeter elements in a Coxeter group (following e.\,g. \cite{casscox}), we can conjugate $tu_0t^{-1}$ by a suitable product of various $u_i(\xi_i)$ ($\xi_i\in k$, $1\leqslant i\leqslant 6$) to obtain $u_0^{-1}$. Explicitly, setting $u:=u_6(-1)u_5(-1)u_6(-1)u_4(-1)u_5(-1)u_6(-1)u_1(1)$ if $\sigma=\id_W$, respectively $u:=u_4(-1)u_6(1)u_1(1)$ in case $\sigma\neq\id_W$, we get $(ut)u_0(ut)^{-1}=u_0^{-1}$ and $u\in\mf U^F$, so $ut\in\mf B^F$.
\end{proof}
For $x\neq x_1, x_2$, we have $\{x,x_1\}=\{x,x_2\}$. We evaluate \eqref{R1bar} at $u_{0}$, using the Lemma:
\begin{align*}
R_{x_1}(u_{0})&=\frac23\rho_{x_2}(u_{0})-\frac13\rho_{x_1}(u_{0})+\sum_{x\in\overline X(W,\sigma)\backslash\{x_1,x_2\}}\{x,x_1\}\rho_x(u_{0}) \\
&=\frac23\rho_{x_2}(u_{0}^{-1})-\frac13\rho_{x_1}(u_{0}^{-1})+\sum_{x\in\overline X(W,\sigma)\backslash\{x_1,x_2\}}\{x,x_2\}\rho_x(u_{0}) \\
&=R_{x_2}(u_{0}).
\end{align*}
This in turn implies that
\[\xi_{x_1}q^3=\xi_{x_1}\chi_1(u_{0})=R_{x_1}(u_{0})=R_{x_2}(u_{0})=\xi_{x_2}\chi_2(u_{0})=\xi_{x_2}q^3,\]
which also equals $\overline{R_{x_2}(u_{0})}$. We deduce $\xi_{x_1}=\xi_{x_2}=\overline\xi_{x_2}$, and then $\xi_{x_1}=\xi_{x_2}\in\{\pm1\}$, since $\xi_{x_2}\overline\xi_{x_2}=1$. Hence we can rewrite \eqref{Fourierrest}:
\begin{equation}\label{rhoxi}
\rho|_{\mf G^F_\mathrm{uni}}=\sum_{\phi\in\Irr(W)}\{\overline x_\rho,x_\phi\}R_{\tilde\phi}|_{\mf G^F_\mathrm{uni}}+\xi\sum_{i=1}^2\{\overline x_\rho,x_i\}\chi_i|_{\mf G^F_\mathrm{uni}}\quad\text{for }\rho\in\Uch(\mf G^F)
\end{equation}
where $\xi:=\xi_{x_1}=\xi_{x_2}$.
\begin{Rem}
We have 
\[R_{\tilde\phi}(u_{0})=\begin{cases}1\quad\text{if }\phi=1_W, \\ 0\quad\text{if }\phi\neq1_W.\end{cases}\]
This follows from Lusztig's algorithm in \cite[\S 24]{LuCS5} and the fact that the Green functions associated to character sheaves considered there coincide with the Green functions arising from Deligne-Lusztig characters by \cite[2.2]{Sh1}. Indeed, the preferred extension of $1_W$ is again $1_W$, so
\[R_{\tilde1_W}=\frac1{|W|}\sum_{w\in W}R_w=1_{\mf G^F}.\]
On the other hand, using the explicitly known Springer correspondence between $\Irr(W)$ and certain pairs $(\mathscr O,\vartheta)$ where $\mathscr O\subseteq\mf G$ is an $F$-stable unipotent class, $u_{\mathscr O}\in\mathscr O^F$, and $\vartheta$ an irreducible character of $A(u_{\mathscr O})$ (in this form, the Springer correspondence is contained in {\sffamily {CHEVIE}} \cite{MiChv}), we see that any non-trivial irreducible character $\phi$ of $W$ belongs to a pair $(\mathscr O,\vartheta)$ such that $\mathscr O\neq\mathscr O_{\mathrm{reg}}$. Since $\mathscr O_{\mathrm{reg}}$ is not contained in the closure of $\mathscr O$ whenever $\mathscr O\neq\mathscr O_{\mathrm{reg}}$, the algorithm in \cite[\S 24]{LuCS5} shows that $R_{\tilde\phi}(u_0)=0$ provided $\phi\neq 1_W$.
\end{Rem}
We can now formulate the result.
\begin{Prop}
The scalar $\xi$ in \eqref{rhoxi} is $+1$, so we get
\[\rho|_{\mf G^F_\mathrm{uni}}=\sum_{\phi\in\Irr(W)}\{\overline x_\rho,x_\phi\}R_{\tilde\phi}|_{\mf G^F_\mathrm{uni}}+\sum_{i=1}^2\{\overline x_\rho,x_i\}\chi_i|_{\mf G^F_\mathrm{uni}}\quad\text{for }\rho\in\Uch(\mf G^F).\]
\end{Prop}
\begin{proof}
Evaluating \eqref{rhoxi} at $u_{0}$ gives
\begin{equation}\label{valreguni}
\rho(u_{0})=\{\overline x_\rho,x_{1_W}\}+\xi q^3\sum_{i=1}^2\{\overline x_\rho,x_i\}.
\end{equation}
To determine $\xi$, we consider the Hecke algebra of the group $\mf G^F$ with its $BN$-pair $(\mf B^F,N_{\mf G}(\mf T)^F)$, that is, the endomorphism algebra
\[\mathcal H_\sigma:=\End_{\mathbb C\mf G^F}(\mathbb C[{\mf G^F}/{\mf B^F}])^\mathrm{opp}\]
(\enquote{opp} stands for the opposite algebra). $\mathcal H_\sigma$ has a $\mathbb C$-basis $\{T_w\mid w\in W^\sigma\}$ where
\[T_w\colon\mathbb C[{\mf G^F}/{\mf B^F}]\rightarrow\mathbb C[{{\mf G}^F}/{\mf B^F}],\quad x\mf B^F\mapsto\sum_{\substack{y\mf B^F\in{\mf G^F}/{\mf B^F} \\ x^{-1}y\in\mf B^F\dot w \mf B^F}}y\mf B^F,\]
for $w\in W^\sigma$. Here, $W^\sigma\cong N_{\mathbf G}(\mathbf T)^F/{\mathbf T^F}$ is a Coxeter group with Coxeter generators $S_\sigma$ consisting of simple reflections corresponding to the orbits of the map $\Pi\rightarrow\Pi$, $\alpha\mapsto\alpha^\dagger$. As already noted in \ref{Not}, we can take $S_{\id_W}:=S=\{s_1,\ldots,s_6\}$ ($s_i=w_{\alpha_i}$, $1\leqslant i\leqslant6$). If $\sigma\neq\id_W$, then $S_\sigma:=\{s_2,s_4,s_3s_5,s_1s_6\}$ gives rise to a Coxeter system $(W^\sigma,S_\sigma)$ of type $F_4$. Denote by $\ell_\sigma\colon W^\sigma\rightarrow\mathbb Z_{\geqslant0}$ the length function of $W^\sigma$ with respect to $S_\sigma$. Then the multiplication in $\mathcal H_\sigma$ is determined by the following equations.
\[T_s\cdot T_w=\begin{cases}\hfil T_{sw}\quad&\text{if}\quad\ell_\sigma(sw)=\ell_\sigma(w)+1 \\ q_sT_{sw}+(q_s-1)T_{w}\quad&\text{if}\quad\ell_\sigma(sw)=\ell_\sigma(w)-1\end{cases}\quad(\text{for }s\in S_\sigma,\, w\in W^\sigma).\]
Here, the $q_s$ ($s\in S_\sigma$) are the parameters of the Hecke algebra $\mathcal H_\sigma$. If $\sigma=\id_W$, we have $q_s=q$ for all $s\in S$ while if $\sigma\neq\id_W$, then $q_{s_2}=q_{s_4}=q$, $q_{s_3s_5}=q_{s_1s_6}=q^2$, see \cite[p.~35]{Lurepchev}, \cite[(7.7)]{LuCoxFrob}. The irreducible characters of $W^\sigma$ naturally parametrise both the isomorphism classes of irreducible modules of $\mathcal H_\sigma$ and the irreducible characters of $\mf G^F$ which are constituents of $\Ind_{\mf B^F}^{\mf G^F}(1_{\mf B^F})$, see \cite[\S 68 and \S 11D]{CR}. Given $\phi\in\Irr(W^\sigma)$, let $V_\phi$ be the  module of $\mathcal H_\sigma$ and $\rho_\phi$ the irreducible character of $\mf G^F$ corresponding to $\phi$. By \cite[3.6]{GCH} and \cite[\S 8.4]{GePf}, we have
\[|O_g\cap\mf B^F\dot w\mf B^F|=\frac{|\mf B^F|}{|C_{\mf G^F}(g)|}\sum_{\phi\in\Irr(W^\sigma)}\rho_\phi(g)\Tr(T_w,V_\phi)\]
for any $g\in\mf G^F$ and $w\in W^\sigma$, where $\dot w\in N_{\mf G}(\mf T)^F$ is a representative of $w$, $O_g$ denotes the conjugacy class of $g$ in $\mf G^F$ and $\Tr(T_w,V_\phi)$ is the trace of the linear map on $V_\phi$ defined by $T_w$. In particular, the number
\[m(g,w):=\sum_{\phi\in\Irr(W^\sigma)}\rho_\phi(g)\Tr(T_w,V_\phi)\]
is non-negative, for any $g\in\mf G^F$ and $w\in W^\sigma$. The character table of $\mathcal H_\sigma$ is contained in {\sffamily {CHEVIE}} \cite{CHEVIE}, so the numbers $\Tr(T_w,V_\phi)$ are known. Choosing $g=u_{0}$, \eqref{valreguni} now reads
\begin{equation}\label{valreguniphi}
\rho_\phi(u_{0})=\{\overline x_{\rho_\phi},x_{1_W}\}+\xi q^3(\{\overline x_{\rho_\phi},x_1\}+\{\overline x_{\rho_\phi},x_2\})\quad\text{for }\phi\in\Irr(W^\sigma).
\end{equation}
Let us first consider the non-twisted case $E_6(q)$, so $\sigma=\id_W$. By \cite[8.7]{Luchars}, $\rho_\phi$ is the unipotent character of $\mf G^F$ corresponding to $\phi$ under the embedding $\Irr(W)\hookrightarrow\overline X(W,\sigma)\leftrightarrow\Uch(\mf G^F)$ (see \eqref{ParUch}, \eqref{IrrWembed}), that is, $\overline x_{\rho_\phi}=x_\phi$ for any $\phi\in\Irr(W)$. Using the notation in \cite[p.~480]{C}, $m(u_{0},w)$ is equal to
\[\Tr(T_w,V_{\phi_{1,0}})+\frac23\xi q^3\bigl(\Tr(T_w,V_{\phi_{80,7}})+\Tr(T_w,V_{\phi_{20,10}})-\Tr(T_w,V_{\phi_{10,9}})-\Tr(T_w,V_{\phi_{90,8}})\bigr).\]
Choosing for $w$ the Coxeter element $w:=s_1s_2s_3s_4s_5s_6$, we obtain
\[0\leqslant m(u_{0},w)=(2\xi+1)\cdot q^6\]
which would be false if $\xi=-1$, so we must have $\xi=1$. \\
Now assume that we are in the twisted case $\leftidx{^2}E_6(q)$ (i.\,e. $\sigma\neq\id_W$). Then the orders in the tables in \cite[p.~480-481]{C} coincide with respect to the parametrisation \eqref{ParUch}, see \cite[1.14-1.16]{LuUniexc}. Setting $w:=s_2s_4(s_3s_5)(s_1s_6)$ (a Coxeter element of $(W^\sigma, S_\sigma)$), we get
\begin{align*}
m(u_0,w)&=\Tr(T_w,V_{\phi_{1,0}})+\frac23\xi q^3\bigl(\Tr(T_w,V_{\phi_{12,4}})-\Tr(T_w,V_{\phi_{6,6}'})-\Tr\bigl(T_w,V_{\phi_{6,6}''}\bigr)\bigr) \\
&=(2\xi+1)\cdot q^6
\end{align*}
and this is non-negative, thus $\xi=1$.
\end{proof}

\subsection*{Acknowledgements}
I thank Meinolf Geck for many comments and hints. This work was supported by DFG SFB-TRR 195.

\bibliographystyle{hacm}

\bibliography{E6p3}

\end{document}